\documentclass{amsart}
\usepackage{latexsym}
\usepackage{graphicx,subfigure}
\usepackage{array}
\usepackage{indentfirst}
\usepackage{epsf}
\usepackage{amsmath,bm}
\usepackage{amsthm}
\usepackage{amsfonts}
\usepackage{epsf}
\usepackage{multicol}
\usepackage[square, comma, sort&compress, numbers]{natbib}
\usepackage{pxfonts}
\usepackage{caption}
\usepackage{epstopdf}
\usepackage{algorithm}
\usepackage{algpseudocode}

\newtheorem{Th}{Theorem}[section]
\newtheorem{Lm}{Lemma}[section]

\newcommand{\bc}{\begin{center}}
\newcommand{\ec}{\end{center}}
\newcommand{\be}{\begin{eqnarray}}
\newcommand{\ee}{\end{eqnarray}}

\newcommand{\ben}{\begin{eqnarray*}}
\newcommand{\een}{\end{eqnarray*}}

\newcommand{\Om}{{\rm\Omega}}

\newcommand{\dx}{\,dx}

\newcommand{\ds}{\,ds}
\newcommand{\cE}{\mathcal{E}}

\newcommand{\cT}{\mathcal{T}}
\newcommand{\R}{\mathbb{R}}

\begin{document}

%%%%%%%%%%%%%%%%%%%%%%%%%%%%%%%%%%%%%%%%%%%%%%%%%%%%%%%%%%%%%%%%%%%%%%%%%%

\title{
A Penalized Crouzeix-Raviart element method for Second Order Elliptic eigenvalue problems
}

\author {Jun Hu}
\address{LMAM and School of Mathematical Sciences, Peking University,
  Beijing 100871, P. R. China.  hujun@math.pku.edu.cn}

\author{Limin Ma}
\address{LMAM and School of Mathematical Sciences, Peking University,
  Beijing 100871, P. R. China. maliminpku@gmail.com}
\thanks{The first author was supported by  NSFC
projects 11271035, 91430213 and 11421101}

\maketitle
\begin{abstract}
In this paper we propose a penalized Crouzeix-Raviart element method for eigenvalue problems of second order elliptic operators. The key idea is to add a penalty term to tune the local approximation property and the global continuity property of the discrete eigenfunctions. The feature of this method is that by adjusting the penalty parameter, the resulted discrete eigenvalues can be in a state of "chaos", and consequently a large portion of them can be reliable and approximate the exact ones with high accuracy. Furthermore, we design an algorithm to select such a quasi-optimal penalty parameter. Finally, we provide numerical tests to demonstrate the performance of the proposed method.

  \vskip 15pt


 \vskip 15pt


\end{abstract}

\section{Introduction}
Finding eigenvalues of operators is important in the mathematical science. Many numerical methods have been used to approximate  eigenvalue problems of partial differential operators, such as finite differences methods, finite element methods and spectral methods. As pointed out in \cite{spectral}, spectral methods can perform extremely well when they are used to approximate eigenvalue problems. Especially, for the 1-D Laplacian operator, two thirds of numerical eigenvalues can be accurate if the Chebyshev pseudo-spectral method is used. However, it is well-known that only a small portion of numerical eigenvalues can be reliable when finite differences and finite element methods are applied. Recently \cite{zhangzhimin} puts forward a quantitative measurement criteria about the number of "trusted" eigenvalues by the finite element approximation of $2m$-th order elliptic eigenvalue problems. In particular, it points out that for 2-D second order elliptic problems, only some earlier eigenvalues can be approximated at a quadratic convergence rate if the linear element is used. Thus, it is not easy for finite element methods to perform well for a large amount of eigenvalues.

The purpose of this paper is to improve accuracy for a large number of eigenvalues. The theories from \cite{AD04,HuHuangShen2014,YangZhangLin2010,huhuanglin} indicate that under certain conditions, the Crouzeix-Raviart element method (CR element method for short hereinafter) produces lower bounds of eigenvalues. The idea herein is to add a penalty term to the discrete bilinear form of the Crouzeix-Raviart element method of second order elliptic eigenvalue problems. We call the resulted method the Penalized Crouzeix-Raviart element method( PCR element method for short hereinafter).  Such a penalty term is able to adjust the continuity of discrete eigenfunctions.  In fact, when the penalty parameter tends to infinity, the discrete eigenfunctions tend to those by the conforming linear element method, which produces upper bounds of eigenvalues. In other words, if the penalty parameter of the PCR element method tends to infinity, the PCR element method becomes the conforming linear element method. This and the theories \cite{AD04,HuHuangShen2014,YangZhangLin2010,huhuanglin} imply that on a given mesh, for some eigenvalue, the exact one can be obtained by appropriately choosing the penalty parameter. Hence,  the PCR element method is  able to  improve accuracy for eigenvalues by tuning its penalty parameter. In particular, if its penalty parameter is selected such that some approximate eigenvalues are lower bounds, and the others are upper bounds  of the exact ones( we call such a state as a "chaos" state), it can be believed that a large amount of approximate eigenvalues admit high accuracy.

In order to accomplish a scheme with approximate eigenvalues in a "chaos " state, we propose an algorithm for designing the penalty parameter. The  main idea is to use the monotonicity property of approximate eigenvalues  on several coarse meshes.
This enables us to  find an interval  such that the PCR element methods taking its two endpoints as the penalty parameters produce lower bounds and upper bounds of eigenvalues, respectively. Then we use a method of bisection to obtain a penalty
parameter such that the approximate eigenvalues are  in a "chaos " state,  see Algorithm \ref{alg:bisection} for more details.  In the following, we call such a parameter a quasi-optimal penalty parameter.

In the case that the aforementioned interval is not found after several elaborate guesses of the penalty parameter are sampled, a quadratic function of the penalty parameter is constructed to fit the average ratio of differences between corresponding approximate eigenvalues on two successive meshes. If there exists a zero of this function, this zero is taken as the penalty parameter; otherwise, the maximum of the intercepts with the $\gamma$-axis of the tangents to this function at the guess points  is taken as a further candidate parameter.

This paper is organized as follows. In Section \ref{notation}, the second order elliptic eigenvalue problem and  the PCR element method are presented, and  some notation are also given. In Section \ref{sectionalg}, an algorithm is proposed to obtain  a quasi-optimal  penalty parameter.  In  Section \ref{sec:numericaltest}, some numerical tests are presented to illustrate the PCR element method.

\section{Notation and Preliminaries}\label{notation}
\subsection{Notation}
Throughout this paper, the standard space, norm, and inner product notation are adopted. Their definitions can be found in \cite{BrennerScott1996}. Suppose that $\Om\in \mathbb{R}^d, d=2 \text{ or }3$ is covered exactly by a shape-regular partition $\cT_h$ into simplices. Denote the set of all interior $(d-1)$-face and boundary $(d-1)$-face of $\cT_h$ by $\cE_h^I$ and $\cE_h^b$, respectively, and $\cE_h=\cE_h^I\cup \cE_h^b$. Let $[\cdot]$ be jumps of piecewise functions over $(d-1)$-face $e$ and $\{\cdot \}$ be averages, namely
	\be
	[v]|_e\ :=\
	\begin{cases}
		v|_K-v|_{K'} , & \text{if the global label of K is bigger,}\cr v|_{K'}-v|_{K} , & \text{if the global label of K' is bigger.}
	\end{cases}
	\ee
$$ \{v\}|_e:=\ \frac{1}{2}(v|_K+v|_{K'})$$
for a piecewise function $v$ and $e=K\cap K'$.

Let $h_K$ denote the diameter of an element $K\in \cT_h$ and $h=\max_{K\in\cT_h}h_K$. Assume that $h_K \simeq h$.  For $x\in K\subset\R^d,\ d=2,3,\ r\in \mathbb{Z}^+$, let $$P_r(K):=\ \sum_{|\alpha|\leq r}c_{\alpha}x^{\alpha}.$$
Throughout the paper, we shall use the symbol $A\lesssim B$ to denote that $A\leq CB$, where $C$ is a positive constant.
\subsection{Penalized Crouzeix-Raviart element method}
Let $\Om\in \R^d,\ d=2,3$ be a bounded polyhedral domain. We consider the following eigenvalue problems of second order elliptic operators:
\begin{align}
-\triangledown \cdot (a\ \triangledown \ u)\ +\ cu\ &=\ \lambda u \qquad\text{in}\ \Om,\label{generaleq01}\\
u\ &=\ 0\qquad\text{on}\ \partial{\Om}.\label{generaleq02}
\end{align}

Then the continuous problem \eqref{generaleq01}-\eqref{generaleq02} can be written in a weak form : Seek $(\lambda,u)\  \in \mathbb{R} \times V$ with $\parallel u \parallel_{0,\Om}\ =\ 1$ such that
	\be\label{conEq}
	a(u,v)\ =\ \lambda (u,v)\ ,\ \forall v \in V,
	\ee
	where $V\ = H^1_0(\Om)$ and the bilinear form is defined as
	\be
	a(u,v)\ =\  \int_{\Om}\ a\triangledown u\cdot \triangledown v \ +\ cuv\dx.
	\ee
The bilinear form $a(u,v)$ is symmetric, bounded, and coercive in the following sense:
\be\label{coeriveCond}
a(u,v)=a(v,u),\ |a(u,v)|\lesssim \parallel u\parallel_{1,\Om}\parallel v\parallel_{1,\Om},\ \parallel v\parallel_{1,\Om}^2\lesssim a(v,v)\forall u,v \in V.
\ee

Under the conditions \eqref{coeriveCond}, the eigenvalue problem \eqref{conEq} has a sequence of eigenvalues
$$0<\lambda_1\leq \lambda_2\leq \lambda_3\leq ...\nearrow +\infty.$$
They satisfy the well-known minimum-maximum principle:
\be\label{minmax}
\lambda_k=\min_{dim V_k=k,V_k\subset V} \max_{v\in V_k, v\neq 0}\frac{a(v,v)}{(v,v)}.
\ee

To improve accuracy for a large amount of eigenvalues by balancing the local approximation property and the global continuity property of discrete eigenfunctions, we propose a penalized Crouzeix-Raviart element method. The discrete space $V_h$ is the standard CR element space, i.e.
\begin{equation}\label{disV}
	\begin{split}
	V_h\ :\  =& \{\  v\in L^2(\Om)  : \  v|_K\  \in \  P_1(K) \text{ for each }K\  \in \cT_h,\int_e[v]\ds \ =\ 0,\ \forall e\in \cE_h^I,\\
      &\int_e v\ds \ =\ 0,\ \forall e\in \cE_h^b \}.\\
	\end{split}
\end{equation}
The bilinear form $\ a_h^{\gamma}:V_h\ \times \ V_h \rightarrow \mathbb{R}$ is defined by:
\be\label{ah}
	a_h^{\gamma}(u,v)\  :=\  \sum_{K\in \cT_h} \int_Ka\triangledown u\cdot \triangledown v+\ cuv \dx +\gamma \sum_{e\in \cE_h^I} \frac{1}{|e|^{2-d/2}}\int_e [u][v]\ds,
	\ee
where the gradient operator $\nabla$ is defined elementwise. The second term in equation \eqref{ah} is the so-called penalty term and $\gamma\geq 0$ is a penalty parameter.

The corresponding finite element approximation of problem \eqref{conEq} is : Find $(\lambda_h^{\gamma},\ u_h^{\gamma}) \in \R\rm \times \rm V_h$, such that
\be\label{disEq}
a_h^{\gamma}(u_h^{\gamma},v)\  =\  \lambda_h^{\gamma} (u_h^{\gamma},v)\quad \text{with} \quad \parallel u_h ^{\gamma}\parallel_{0,\Om}\ =\ 1\quad \forall v\in V_h.
\ee	

If the parameter $\gamma=0$, the PCR element method reduces to the CR element method, and will approximate the exact eigenvalues from below under some conditions, see, for instance, \cite{AD04,HuHuangShen2014,YangZhangLin2010,huhuanglin}. It is shown by the numerical tests in Section \ref{sec:2dtest} that the PCR element method is insensitive to penalty parameters, namely, the PCR element method with the penalty parameter near to the optimal one is able to  equally achieve high accuracy for eigenvalues.

Next, we list two properties of the PCR element method in the following theorems.

\begin{Th}\label{monotone}
Let $\lambda_i$ be the $i$-th eigenvalue of the problem \eqref{conEq}, $\lambda_{i,h}^{\gamma}$ be the approximation to $\lambda_i$ by the PCR element method with  the penalty parameter $\gamma$, then $\lambda_{i,h}^{\gamma}$ is monotonically increasing along with the penalty parameter $\gamma$.
\end{Th}
\begin{proof}
Suppose $0\leq \gamma_1<\gamma_2$, the corresponding sequences of discrete eigenvalues are denoted by $\{\lambda^{\gamma_1}_i\}_{i=1}^N$ and $\{\lambda^{\gamma_2}_i\}_{i=1}^N$, where $N=\text{dim} V_h$. Then, according to the discrete minimum-maximum principle,
\be
\lambda_k^{\gamma_1}=\min_{\text{dim} V_k=k,V_k\subset V_h} \max_{v\in V_k, v\neq 0}\frac{a_h^{\gamma_1}(v,v)}{(v,v)}\leq \min_{\text{dim} V_k=k,V_k\subset V_h} \max_{v\in V_k, v\neq 0}\frac{a_h^{\gamma_2}(v,v)}{(v,v)}=\lambda_k^{\gamma_2}
\ee
for any positive integer $k\leq N$, which completes the proof.
\end{proof}

\begin{Th}\label{infityconforming}
Let $\lambda_i$ be the $i$-th eigenvalue of the problem \eqref{conEq}, $\lambda_{i,h}^{\gamma}$ be the approximation to $\lambda_i$ by the PCR element method with  the penalty parameter $\gamma$, $\lambda_{i,h}^{P_1}$ be the approximation to $\lambda_i$ by the conforming linear method. Then
\begin{equation}
\lim_{\gamma\rightarrow +\infty} \lambda_{i,h}^{\gamma}=\begin{cases}
\lambda_{i,h}^{P_1} & i\leq \text{ dim }V_h^c\\
+\infty & i>\text{ dim }V_h^c
\end{cases},
\end{equation}
where $$V_h^c=\{v\in L^2(\Om): v|_K \in P_1(K), \text{for any }K\in\cT_h, v\text{ is continuous on interior faces}, v|_{\partial \Om}=0 \}$$ is the conforming linear element space.
\end{Th}
\begin{proof}
Define the sets $\mathcal{V}_k=\{V_{k}\subset V_h: \text{dim} V_{k}=k\}, \mathcal{V}_k^c=\{V_{k}\subset V_h^c:\text{dim} V_{k}=k\}$, $\mathcal{V}_k^{nc}=\{V_{k}\nsubseteq V_h^c: \text{dim} V_{k}=k\}$. Then, the set $\mathcal{V}_k$ has the following decomposition:
\begin{equation}\label{cncdecom}
\mathcal{V}_k=\mathcal{V}_k^c \cup \mathcal{V}_k^{nc}.
\end{equation}
Due to the discrete minimum-maximum principle,
\begin{equation}\label{minmaxdecomposition}
\begin{split}
\lambda_{k,h}^{\gamma}&=\min_{V_k\in \mathcal{V}_k} \max_{v\in V_k, v\neq 0}\frac{a_h^{\gamma}(v,v)}{(v,v)}\\
&=\min \big(\min_{V_k\in \mathcal{V}_{k}^c}\max_{v\in V_k, v\neq 0}\frac{a_h^{\gamma}(v,v)}{(v,v)},\min_{V_k\in \mathcal{V}_{k}^{nc}}\max_{v\in V_k, v\neq 0}\frac{a_h^{\gamma}(v,v)}{(v,v)}\big).
\end{split}
\end{equation}

For any $V_{k}\in \mathcal{V}_{k}^{nc}$, there exist nonconforming functions in it, say $u^{nc}$. This implies that
$$\sum_{e\in \cE_h}\frac{1}{|e|}\parallel [u^{nc}]\parallel_{0,e}^2 \neq 0,$$
Therefore, if $\gamma\rightarrow +\infty$,
$$\max_{u\in V_k}\frac{a_h^{\gamma}(u,u)}{(u,u)}\geq \frac{a_h^{\gamma}(u^{nc},u^{nc})}{(u^{nc},u^{nc})}\rightarrow +\infty .$$
As a result,
\begin{equation}\label{ncinfty}
\min_{V_k\in \mathcal{V}_{k,h}^{nc}}\max_{v\in V_k, v\neq 0}\frac{a_h^{\gamma}(v,v)}{(v,v)} \rightarrow +\infty .
\end{equation}

On the other hand, for any $V_{k}\in \mathcal{V}_{k}^c$,
$$\sum_{e\in \cE_h^I}\frac{1}{|e|}\parallel [u^{c}]\parallel_{0,e}^2=0, \forall u^c\in V_{k,h},$$
thus, $\max_{u\in V_k, u\neq 0}\frac{a_h^{\gamma}(u,u)}{(u,u)}$ is a constant function with respect to $\gamma$. This is to say that
\begin{equation}\label{cinfty}
\min_{V_k\in \mathcal{V}_{k}^c}\max_{v\in V_k, v\neq 0}\frac{a_h^{\gamma}(v,v)}{(v,v)} \equiv \lambda_{k,h}^{P_1}, \forall \gamma>0 .
\end{equation}

Substituting \eqref{ncinfty}, \eqref{cinfty} into \eqref{minmaxdecomposition}, if $k\leq \text{ dim }V_h^c$,  then $\mathcal{V}_k^c \neq \varnothing$,
$$
\lim_{\gamma\rightarrow +\infty} \lambda_{i,h}^{\gamma}=\lambda_{i,h}^{P_1},
$$
whereas, if $k>\text{ dim }V_h^c$, $\mathcal{V}_k^c = \varnothing$, thus
$$
\lim_{\gamma\rightarrow +\infty} \lambda_{i,h}^{\gamma}=+\infty .
$$
Thus the proof is completed.
\end{proof}

As is known, when $\gamma=0$, the PCR element method produces lower bounds, namely $\lambda_{k,h}^0<\lambda_k$, while, when $\gamma\rightarrow +\infty$, according to Theorem \ref{infityconforming}, $\lambda^{\gamma}_k$ tends to be an upper bound of $\lambda_k$, namely, $\lambda^{\gamma}_{k,h}>\lambda_k$. Then, Theorem \ref{monotone} indicates that for any given eigenvalue $\lambda_k$, there must exists a corresponding penalty parameter $\gamma_k^{\ast}$ that satisfies
$$\lambda_{k,h}^{\gamma_k^{\ast}}=\lambda_k.$$

Next, we consider the continuity and coercivity of the bilinear form $a_h^{\gamma}(\cdot,\cdot)$. A standard argument for nonconforming finite element methods, see, for instance, \cite{BrennerScott1996,ShiWang}, proves that the CR element space $V_h$ and bilinear form $a_h^{\gamma}(\cdot ,\ \cdot)$ have the following properties.
	\begin{itemize}
		\item[(H1)] $| \cdot |_h\ :=\ a_h^{\gamma}(\cdot,\ \cdot)^{1/2}$ is a norm over the discrete space $V_h$;
		\item[(H2)] Suppose $v\ \in\ V\cap H^2(\Om)$ , then $$\inf_{v_h\in V_h} |v\ -\ v_h|_{1,h} \ \lesssim h|v|_{2,\Om};$$
$$ \sup_{0\neq v_h\in V_h}\frac{|a_h^{\gamma}(v,v_h)-a(v,v_h)|}{|v_h|_{h}^2}\ \lesssim h |v|_{2,\Om}.$$		
	\end{itemize}
	We have the error estimate in Lemma \ref{esti0}, we refer to \cite{HuHuangShen2014} for details on this result.
	\begin{Lm}\label{esti0}
		Suppose $(\lambda, u)$ is the solution to problem \eqref{generaleq01}-\eqref{generaleq02}, and $(\lambda_h^{\gamma},u_h^{\gamma})$ is the FEM solution to problem \eqref{disEq}, there exists the following estimate for eigenvalues $\lambda$ and their corresponding functions $u$ :
		$$|\lambda -  \lambda_h^{\gamma}|\  +\ \parallel u - u_h^{\gamma}\parallel_{0,\Om}\ +\ h|u -u_h^{\gamma}|_{h}\ \lesssim h^2|u|_{2,\Om}$$
provided that $u\in H^2(\Om)\cap H^1_0(\Om)$.
	\end{Lm}

\section{Algorithm for Quasi-Optimal Penalty Parameters}\label{sectionalg}
The penalty parameter plays an important role in the performance of the PCR element method. In this section, we design an algorithm to find a quasi-optimal penalty parameter that the corresponding PCR element method produces upper bounds of some eigenvalues and lower bounds of the others, namely, the eigenvalues by the corresponding PCR element method are in a "chaos" state. The main idea is to find an interval $[\gamma_l, \gamma_u]$, such that the approximate eigenvalues obtained by the PCR element methods with $\gamma_l$ and $\gamma_u$ as the penalty parameters are lower bounds and upper bounds of eigenvalues, respectively. Then an application of a bisection-type method is able to achieve a quasi-optimal penalty parameter $\gamma^{\ast}$.

For simplicity of presentation,  we classify penalty parameters into three types, denoted by Type 1, Type 2 and Type 3.  For a   parameter of Type 1, most (up to some ratio to be specified later) of the approximate eigenvalues by the corresponding PCR element method are upper bounds of the exact ones.  For a   parameter of Type 2, most (up to some ratio to be specified later) of the approximate eigenvalues by the corresponding PCR element method are lower bounds of the exact ones.  Type 3 is the remaining part, namely, for any   parameter of Type 3,  the approximate eigenvalues by the corresponding PCR element method are in a "chaos" state.

Before we illustrate how to tell to which type a specialized penalty parameter belongs, we first introduce some further notation. We arrange the eigenvalues of the eigenvalue problem \eqref{conEq} as
\begin{equation}\label{exactseq}
0<\lambda_1\leq\lambda_2\leq \lambda_3\leq \nearrow +\infty .
\end{equation}
Let $\{\cT_i\}_{i=1}^k$ be a shape regular family of conforming simplicial triangulations of the computational domain ¦¸, obtained by successive quasi-uniform refinement of the initial mesh $\cT_1$.
Denote the total number of eigenvalues on $\cT_i$ by $N_i$, we compute the first $M_i$ eigenvalues of the $N_i$ eigenvalues. Then, we define the ratio
\be\label{etadef}
\eta=M_i/N_i.
\ee
In the paper,  we keep the ratio $\eta$ fixed on  successive quasi-uniform refinement $\cT_i$ of the initial mesh $\cT_1$.

Denote the approximation to $\lambda_j$ on $\cT_i$ by $\lambda_{i,j}^{\gamma}, i=1,2,...,k$, sequences
$$ \wedge_j^{\gamma}=\{ \lambda_{1,j}^{\gamma},\ \lambda_{2,j}^{\gamma}, ...\ ,\lambda_{k,j}^{\gamma}\},$$
and the difference of the approximations to $\lambda_j$ on the meshes $\cT_i, \cT_{i+1}$ by
\[D_{ij}=\begin{cases}
        \lambda_{i+1,j}^{\gamma}-\lambda_{ij}^{\gamma}& j=1,...,M_i\\
        0 & j>M_i
\end{cases}.\]
If $D_{ij}>0$, the approximation to $\lambda_j$  increases from $\cT_i$ to $\cT_{i+1}$, otherwise, the approximation to $\lambda_j$ deceases. This motivates us to define the following parameter
\be\label{betadef}
\beta (\gamma,\eta)=\frac{\sum_{i=1}^{k-1}\sum_{j=1}^{M_i} sign(D_{ij})}{\sum_{i=1}^{k-1}M_i}=\frac{\sum_{i=1}^{k-1}\sum_{j=1}^{\eta N_i} sign(D_{ij})}{\sum_{i=1}^{k-1}\eta N_i}.
\ee
The range of $\beta (\gamma,\eta)$ is $[-1,1]$. When all sequences $\wedge_j^{\gamma},j=1,...,M_i$ are strictly monotonically increasing, $\beta (\gamma,\eta)=1$; when all sequences are strictly monotonically decreasing, $\beta(\gamma,\eta)=-1$; otherwise, $\beta (\gamma,\eta)\in (-1,1)$. Thus,  the parameter $\beta(\gamma,\eta)$ measures  the monotonous properties of the first $\eta$ percent of the discrete eigenvalues  on the two successive  meshes. Therefore,  the more  the approximate eigenvalues are in the state of  ``chaos", the closer $\beta (\gamma,\eta)$ is to zero.  Thus,  to achieve high accuracy for a large amount of eigenvalues is to design a penalty parameter such that the corresponding parameter $\beta(\gamma,\eta)$ is close to zero.

We introduce a criterion $\textbf{tol}$ to classify penalty parameters into three types. If $\beta(\gamma,\eta)\leq -\textbf{tol}$, we classify $\gamma$ to Type 1, namely, most of the approximate eigenvalues by the corresponding PCR element method are upper bounds of the exact ones; if $\beta(\gamma,\eta)\geq \textbf{tol}$, we classify $\gamma$ to Type 2, namely, most of the corresponding approximate eigenvalues are lower bounds; otherwise, we classify it to Type 3, namely, the approximate eigenvalues are in a state of "chaos". For a given criterion $\textbf{tol}$, we can use the resulted $\beta(\gamma,\eta)$ to tell which type the penalty parameter $\gamma$ belongs to.

In order to find a penalty parameter $\gamma$ such that the corresponding $\beta(\gamma,\eta)$ is near to zero, we need to select an interval $[\gamma_l,\gamma_u]$ so that $\gamma_l$ belongs to Type 2 and $\gamma_u$ belongs to Type 1. Since  the PCR element method produces lower bounds of eigenvalues when $\gamma=0$ (Type 2),  we only need to choose a penalty parameter $\gamma_u$ which belongs to Type 1. To this end, we make an initial guess $\gamma_1$ of $\gamma_u$, and apply the following Algorithm \ref{alg:gammaguess} to obtain a penalty parameter of Type 1 or a sequence of parameters.

\begin{algorithm}[htb!]
\caption{Select a  penalty parameter of Type 1}
\label{alg:gammaguess}
\begin{algorithmic}[1]
\Require Given an initial guess $\gamma_1$, a positive integer $L\geq 2$, and a multiplicative constant $\rho>1$ ;
\State Let  $i=1;$
\While {$i\leq L$}
\If {$\gamma_i$ belongs to Type 2}
\State {  $ \gamma_{i+1}=\rho\gamma_i,\ i=i+1;$}
\Else
\State {$\gamma_i\rightarrow \hat{\gamma}_u$, break;}
\EndIf
\EndWhile
\Ensure  If $i\leq L$, output $\hat{\gamma}_u$; otherwise, output $\gamma_i, i=1,\cdots,L$.
\end{algorithmic}
\end{algorithm}

If the output $\hat{\gamma}_u$ from Algorithm \ref{alg:gammaguess} is of  Type 1,  we take it as the $\gamma_u$; if it is of  Type 3, we take it as a desirable quasi-optimal penalty parameter $\gamma^{\ast}$ directly. For the  remaining case, we get a sequence of penalty parameters of  Type 2. In such a case, we denote $\gamma_0=0$ and use $\lambda_{k-1,j}^{\gamma_i}, \lambda_{k,j}^{\gamma_i}, i=0, \cdots, L$, which have been computed in Algorithm \ref{alg:gammaguess}, to define discrete average relative error
$$
\Delta\lambda(\gamma_i)=\frac{1}{M_{k-1}}\sum_{j=1}^{M_{k-1}}\frac{\big |\lambda_{k,j}^{\gamma_i}-\lambda_{k-1,j}^{\gamma_i}\big |}{\lambda_{k,j}^{\gamma_i}},\ i=0,1,...,L.
$$
We fit $ \Delta\lambda(\gamma)$ with a quadratic function with respect to the penalty parameter $\gamma$ by  the least square method. It is observed from the numerical results in Table \ref{tab:15}-\ref{tab:60} below that when the penalty parameter is fixed, the average relative error of the first $\eta$ percent of the approximate eigenvalues by the PCR element method is almost a constant as the mesh varies ($\eta$ is defined by \eqref{etadef}). Therefore, it is reasonable to fit the discrete average relative error by a function of penalty parameters.

If the penalty parameter $\gamma$ is of Type 1, $\Delta \lambda (\gamma)$ should be negative; if it is of Type 2,  $\Delta \lambda (\gamma)$ should be positive. In the theory, the fitted quadratic function should be monotonically decreasing on $[0,+\infty )$ and has a positive zero.  However, in numerical tests, it is possible that there does not exist a zero  for
such a quadratic function. Thus, in our method,  if there exist(s)  zero(s),  we take the bigger one as $\hat{\gamma}$; otherwise we denote the zero of the tangent at $\gamma_i$ by $\hat{\gamma_i}$, and take the maximum of $\hat{\gamma_i}, i=0,\cdots,L$ as $\hat{\gamma}$.

If the parameter $\hat{\gamma}$ is of  Type 1, then we take $\gamma_u=\hat{\gamma}$; otherwise it is of Type 2 or Type 3. For the latter case,  the eigenvalues by the PCR element method with penalty parameter $\hat{\gamma}$ are  more accurate than those produced by the original CR element method, which motivate us to  select it as a desirable penalty parameter $\gamma^{\ast}$.

In the end, we  use a bisection-type method in Algorithm \ref{alg:bisection} to obtain a desirable penalty parameter $\hat{\gamma}^{\ast}$ from an interval $[\gamma_l,\gamma_u]$.
\begin{algorithm}[htb!]
\caption{Select a quasi-optimal penalty parameter}
\label{alg:bisection}
\begin{algorithmic}[1]
\Require Given an interval $[\gamma_l, \gamma_u]$, where $\gamma_l$ belongs to Type 2 and $\gamma_u$ belongs to Type 1, and a stopping criterion $ \epsilon$;
\While {$|\gamma_u-\gamma_l|> \epsilon$}
\State $(\gamma_u+\gamma_l)/2\rightarrow \hat{\gamma}^{\ast}$;
\If {$\hat{\gamma^{\ast}}$ belongs to Type 3}
\State break;
\Else\If {$\hat{\gamma^{\ast}}$ belongs to Type 2} {$ \hat{\gamma}^{\ast}\rightarrow \gamma_l$;}
\Else {  $ \hat{\gamma^{\ast}}\rightarrow \gamma_u$;}
\EndIf
\EndIf
\EndWhile
\Ensure  A quasi-optimal penalty $\hat{\gamma}^{\ast}$.
\end{algorithmic}
\end{algorithm}

To get a better penalty parameter, we need to improve Algorithm \ref{alg:bisection}. First, we use the parameter $\hat{\gamma}^{\ast}$ from the Algorithm \ref{alg:bisection} to divide the interval $[\gamma_l,\gamma_u]$ into two parts $[\gamma_l, \hat{\gamma}^{\ast}]$ and $[\hat{\gamma}^{\ast},\gamma_u]$. Second, we use the bisection-type method to $[\hat{\gamma}^{\ast},\gamma_u]$ to obtain the smallest penalty parameter which belongs to Type 1, denoted by $\gamma_u^{\ast}$, and to $[\gamma_l, \hat{\gamma}^{\ast}]$ to obtain the largest penalty parameter which belongs to Type 2, denoted by $\gamma_l^{\ast}$. In the end, we take $\gamma^{\ast}=\frac{\gamma_l^{\ast}+\gamma_u^{\ast}}{2}$ as the final quasi-optimal penalty parameter.

We have to stress that, as is shown by the numerical tests in Section \ref{sec:numericaltest},  the  quasi-optimal penalty parameter $\gamma^\ast$ computed by the algorithm herein is robust and insensitive with respect to these meshes $\cT_1,\cdots,\cT_k$ when they are not that coarse, the criterion $\textbf{tol}$ and the initial guesses $\gamma_i, i=0,\cdots,L$.

\section{Numerical Tests}\label{sec:numericaltest}
In this section, four numerical results are presented to illustrate the performance of the PCR element method for eigenvalue problems of second order elliptic operators. The first three examples are for 2-dimension, the last example is for 3-dimension. In the first two examples, eigenvalue problems of the Laplacian operator are considered, they are carried out sequentially on a square domain $\Om = (0,1)^2$, and a L-shaped domain $\Om=(-1,1)^2/[0,1)\times (-1,0]$. The third one is carried out on a square domain $\Om = (0,1)^2$, solving eigenvalues of a general second order elliptic operator. The last one is carried out on a cubic domain $\Om=(0,1)^3$ for eigenvalues of the Laplacian operator.

The sequence of exact eigenvalues \eqref{exactseq} is defined in Section \ref{sectionalg}, correspondingly, the sequence of numerical eigenvalues is denoted by:
\be
0\ <\ \lambda_{h,1}\ \leq \lambda_{h,2}\ \leq \ \lambda_{h,3} \ \leq \ldots \nearrow +\infty .
\ee
The relative error of the $i$-th approximate eigenvalue on $\cT_k$ is denoted by $$e_{k,i}=\frac{\lambda_i -\ \lambda_{k,i}}{\lambda_i}$$
and the average relative error of the first $\eta$ percent of eigenvalues by
\be\label{Edef}
E_{\eta,k}=\frac{1}{ceil(M_k)}\sum_{i=1}^{ceil(M_k)} |e_{k,i}|
\ee
where $M_k$ is the number of the first $\eta$ percent of eigenvalues on mesh $\cT_k$ and $ceil(M_k)$ denotes the smallest integer larger than $M_k$.

\subsection{}\label{sec:2dtest}

We consider the following eigenvalue problem
\begin{equation}
\begin{split}
-\Delta u\ &=\ \lambda u\  \text{ in}\ \Om ,\\
u\ &=\ 0\  \text{ on}\ \partial\Om ,
\end{split}
\end{equation}
on domain $\Om =(0,1)^2$, and partition the domain by uniform triangles. In the computation, the level one mesh consists of two right triangles, obtained by cutting the unit square with a north-east line. Each mesh is refined into a half sized mesh uniformly, to get a higher level mesh. The exact eigenvalues are known as
$$\lambda=(m^2+n^2)\pi^2 ,\ m,\ n\ \text{are positive integers}.$$

\begin{figure}[!ht]
\centering
\includegraphics[width=13cm,height=12cm]{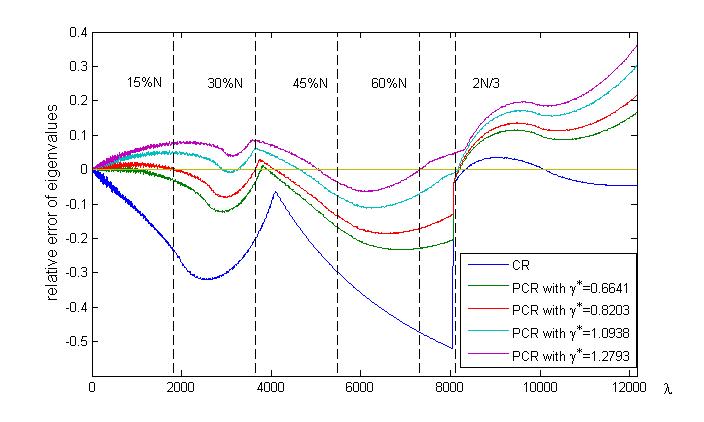}
\caption{\footnotesize{Comparison of the relative errors $\frac{\lambda_h-\lambda}{\lambda}$ by the CR element method and the PCR element method with $\gamma^\ast$ on $\cT_7$. }}
\label{fig:alleig}
\end{figure}

In our computation, we consider four cases where the ratios of eigenvalues we investigate are $15\%$, $30\%$, $45\%$, and $60\%$, denoted by Case 1, Case 2, Case 3 and Case 4, respectively. To select quasi-optimal penalty parameters for these four cases, we let $\eta=10\%$, $20\%$, $30\%$, and $40\%$ in \eqref{betadef}, respectively. In the improved Algorithm \ref{alg:bisection}, we take $k=5$, $\textbf{tol}=0.8$, $\epsilon=0.01$.   In particular, this yields $\gamma^\ast=0.6641$, $0.8203$, $1.0938$, $1.2793$ for these four cases, respectively.

We depict the relative errors by the CR element method and the PCR element methods with $\gamma^\ast=0.6641$, $0.8203$, $1.0938$, $1.2793$ in Figure \ref{fig:alleig}, where the five vertical lines are $x=15\%N$, $30\%N$, $45\%N$, $60\%N$, $2N/3$ with $N=12160$ the total number of eigenvalues on mesh $\cT_7$. As is shown in Figure \ref{fig:alleig}, the relative errors of the approximate eigenvalues by the PCR element methods in these four cases are smaller than $ 3.4\%$, $8.2\%$, $ 7.6\%$ and $ 8.8\%$, respectively, improved from $23\%$, $31\%$, $31\%$ and $47\%$ of those by the CR element method, respectively. As a more precise comparison, the relative error of the $820$-th eigenvalues by the CR element method is almost $10\%$, which is bigger than those of all the first $7296$ eigenvalues by the PCR element method on the aforementioned mesh.

\begin{figure}[!ht]
\centering
\subfigure
{\includegraphics[width=6cm,height=6cm]{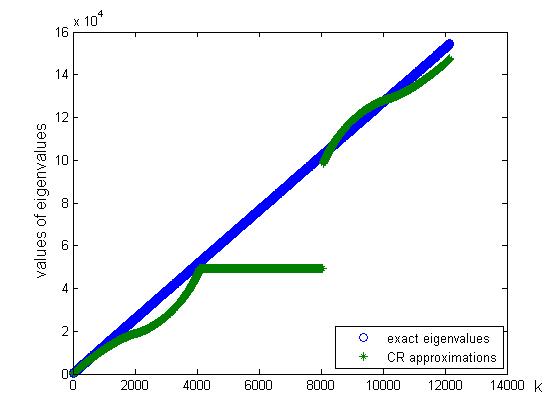}}
\subfigure
{\includegraphics[width=6cm,height=6cm]{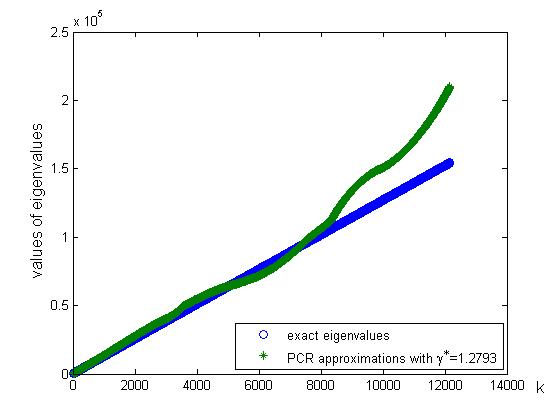}}
\caption{\footnotesize{Comparison of the approximate eigenvalues by the CR element method and the PCR element method with $\gamma^\ast$ on $\cT_7$.}}
\label{fig:exactcompare}
\end{figure}

We plot the exact eigenvalues and the approximate ones by the CR element method and the PCR element method with $\gamma^\ast=1.2793$ in Figure \ref{fig:exactcompare}. Note that the second third of the approximate eigenvalues by the CR element method are almost the same. This fact brings about remarkable errors for these eigenvalues, while the PCR element method with $\gamma^\ast=1.2793$ avoids the emergence of this situation, and approximates the eigenvalues with high accuracy.

\begin{table}[htb!]
\footnotesize
    \caption{\footnotesize{Average relative errors of the first $\eta=15\%$ eigenvalues by the CR element method, conforming linear element method and the PCR element method with different penalty parameters on each mesh. $M_i$ is the number of the first $15\%$ eigenvalues by the CR and the PCR element method on $\cT_i$. }}
    	\centering
	\begin{tabular}{c|ccccccc}
		\hline
		$h$   & $1$ & $\frac{1}{2}$ &  $\frac{1}{2^2}$ &  $\frac{1}{2^3}$ & $\frac{1}{2^4}$&  $\frac{1}{2^5}$ & $\frac{1}{2^6}$\\\hline
        $M_i$ & 1 & 2 & 6& 27&111 & 452&1824 \\\hline
		$CR$   & 0.2159 & 0.2273 & 0.1306 & 0.1238 & 0.1149 & 0.1118 & 0.1098 \\\hline
        $P_1$   &-&0.6211 & 0.2131 & 0.1454 & 0.1138 &0.1059  & 0.1016\\\hline
		$\gamma=0.5$ & 0.3172 & 0.1496 & 0.0416 & 0.0335 & 0.0286 & 0.0269 & 0.0259 \\\hline
        $\gamma=0.6$  & 0.3374 & 0.1344 & 0.0286 & 0.0219 & 0.0178 & 0.0156 & 0.0143 \\\hline
        $\gamma^{\ast}=0.6641$  & 0.3504 & 0.1247& 0.0229&0.0161&0.0126&0.0102&0.0085 \\\hline
        $\gamma=0.7$  & 0.3577 & 0.1193 & 0.0199 & 0.0140 & 0.0108 & 0.0081 & 0.0064 \\\hline
        $\gamma=0.8$  & 0.3780 & 0.1042 & 0.0146 & 0.0121 & 0.0105 & 0.0085 & 0.0082 \\\hline
        $\gamma=0.9$  & 0.3982 & 0.0893 & 0.0121 & 0.0154 & 0.0173 & 0.0174 & 0.0174 \\\hline
	\end{tabular}%
	\label{tab:15}%
\end{table}%

\begin{table}[htb!]
\footnotesize
\caption{\footnotesize{Average relative errors of the first $\eta=30\%$ eigenvalues by the CR element method, conforming linear element method and the PCR element method with different penalty parameters on each mesh. $M_i$ is the number of the first $30\%$ eigenvalues by the CR and the PCR element method on $\cT_i$.}}
	\centering
	\begin{tabular}{c|ccccccc}
		\hline
		$h$   & $1$ & $\frac{1}{2}$ &  $\frac{1}{2^2}$ &  $\frac{1}{2^3}$ & $\frac{1}{2^4}$&  $\frac{1}{2^5}$ & $\frac{1}{2^6}$\\\hline
        $M_i$ & 1 & 3 & 12& 53&221 & 903&3648 \\\hline
		CR  & 0.2159 & 0.2793 & 0.2392 & 0.2181 & 0.2068 & 0.2006 & 0.1973 \\\hline
        $P_1$  &-& 0.6211 & 0.2921 & 0.2313 & 0.2108 &0.2022 & 0.1980\\\hline
        $\gamma=0.7$ & 0.3577 & 0.1316 & 0.0797 & 0.0607 & 0.0497 & 0.0432 & 0.0397 \\\hline
        $\gamma=0.8$  & 0.3780 & 0.1108 & 0.0643 & 0.0470 & 0.0361 & 0.0301 & 0.0273 \\\hline
        $\gamma^{\ast}=0.8203$  &0.3821&0.1066&	0.0620&0.0447&0.0340&	0.0282& 0.0256\\\hline
        $\gamma=0.9$  & 0.3982 & 0.0900 & 0.0538 & 0.0376 & 0.0284 & 0.0236 & 0.0218 \\\hline
        $\gamma=1.0$  & 0.4185 & 0.0734 & 0.0495 & 0.0336 & 0.0266 & 0.0240 & 0.0236 \\\hline
        $\gamma=1.1$  & 0.4388 & 0.0574 & 0.0456 & 0.0348 & 0.0305 & 0.0306 & 0.0316 \\\hline
	\end{tabular}%
    \label{tab:30}%
\end{table}%

\begin{table}[htb!]
\small
\caption{\footnotesize{Average relative errors of the first $\eta=45\%$ eigenvalues by the CR element method, conforming linear element method and the PCR element method with different penalty parameters on each mesh. $M_i$ is the number of the first $45\%$ eigenvalues by the CR and the PCR element method on $\cT_i$. }}
	\centering
	\begin{tabular}{c|ccccccc}
		\hline
	    $h$   & $1$ & $\frac{1}{2}$ &  $\frac{1}{2^2}$ &  $\frac{1}{2^3}$ & $\frac{1}{2^4}$&  $\frac{1}{2^5}$ & $\frac{1}{2^6}$\\\hline
        $M_i$ & 1 & 4 & 28&80 & 332 & 1354&5472 \\\hline
		CR   & 0.2159 & 0.3075 & 0.2559 & 0.2246 & 0.2057 & 0.1957 & 0.1906 \\\hline
        $P_1$   &-& 0.6211 & 0.3919 & 0.3223 & 0.2983 &0.2908 & 0.2877\\\hline
		$\gamma=0.8$ & 0.3780 & 0.1203 & 0.0865 & 0.0625 & 0.0470 & 0.0393 & 0.0366 \\\hline
        $\gamma=0.9$  & 0.3982 & 0.0972 & 0.0724 & 0.0500 & 0.0363 & 0.0312 & 0.0297 \\\hline
        $\gamma=1.0$  & 0.4185 & 0.0771 & 0.0624 & 0.0416 & 0.0311 & 0.0288 & 0.0285 \\\hline
        $\gamma^{\ast}=1.0938$  & 0.4375&	0.0587&0.0531&0.0374&0.0310&0.0314& 0.0324\\\hline
        $\gamma=1.1$  & 0.4388 & 0.0575 & 0.0525 & 0.0373 & 0.0312 & 0.0317 & 0.0329 \\\hline
        $\gamma=1.2$  & 0.4590 & 0.0407 & 0.0428 & 0.0365 & 0.0373 & 0.0409 & 0.0432 \\\hline
	\end{tabular}%
    \label{tab:45}%
\end{table}

\begin{table}[htb!]
\small
\caption{\footnotesize{Average relative errors of the first $\eta=60\%$ eigenvalues by the CR element method, conforming linear element method and the PCR element method with different penalty parameters on each mesh. $M_i$ is the number of the first $60\%$ eigenvalues by the CR and the PCR element method on $\cT_i$. }}
	\centering
	\begin{tabular}{c|ccccccc}
		\hline
		$h$   & $1$ & $\frac{1}{2}$ &  $\frac{1}{2^2}$ &  $\frac{1}{2^3}$ & $\frac{1}{2^4}$&  $\frac{1}{2^5}$ & $\frac{1}{2^6}$\\\hline
        $M_i$ & 1 & 5 & 24& 106&442 & 1805&7296 \\\hline
		CR   & 0.2159 & 0.2514 & 0.2972 & 0.2703 & 0.2538 & 0.2452 & 0.2409 \\\hline
        $P_1$   &-& 0.6211 & 0.4393 & 0.3903 & 0.3694 &0.3633 & 0.3604\\\hline
		$\gamma=0.9$ & 0.3982 & 0.0869 & 0.0904 & 0.0761 & 0.0656 & 0.0615 & 0.0601 \\\hline
        $\gamma=1.0$  & 0.4185 & 0.0725 & 0.0750 & 0.0626 & 0.0547 & 0.0527 & 0.0523 \\\hline
        $\gamma=1.1$  & 0.4388 & 0.0585 & 0.0603 & 0.0521 & 0.0477 & 0.0479 & 0.0486 \\\hline
        $\gamma=1.2$  & 0.4590 & 0.0467 & 0.0487 & 0.0445 & 0.0453 & 0.0478 & 0.0494 \\\hline
        $\gamma^{\ast}=1.2793$  & 0.4751&0.0458&0.0417&0.0425&0.0460&0.0491&0.0508\\\hline
        $\gamma=1.3$  & 0.4793 & 0.0468 & 0.0407 & 0.0424 & 0.0466 & 0.0497 & 0.0513 \\\hline
	\end{tabular}%
    \label{tab:60}%
\end{table}%
We compare the PCR element methods with various penalty parameters with the original CR element method and the conforming linear element method in Table \ref{tab:15} - \ref{tab:60} for the four aforementioned cases respectively. As the mesh size decreases from $1$ to $\frac{1}{2^6}$, the corresponding total numbers of eigenvalues by the CR element method and the PCR element methods are $1 $, $8 $, $40 $, $176 $, $736 $, $3008 $, $12160$. In other words, if the ratio $\eta=60\%$, $1$, $5 $, $24 $, $106 $, $442 $, $1805 $, $7296$ eigenvalues are investigated, respectively. We observe that the CR element method performs better than the conforming linear element method on most cases. In the following, we only compare the PCR element method with the original CR element method.

On the mesh $\cT_7$, for Case 1, the average relative error of the first $1824$ eigenvalues by the PCR element method with the quasi-optimal parameters $\gamma^\ast$ computed by Algorithm \ref{alg:bisection} are smaller than $0.9\%$, improved from $10\%$ of that by the CR element method; for Case 2, the average relative error of the first $3648$ eigenvalues by the PCR element method are smaller than $2.6\%$, improved from $19\%$ of that by the CR element method; for Case 3, the average relative error of the first $5472$ eigenvalues by the PCR element method are smaller than $3.3\%$, improved from $19\%$ of that by the CR element method; and for Case 4, the average relative error of the first $7296$ eigenvalues by the PCR element method are smaller than $5.1\%$, improved from $20\%$ of that by the CR element method. It indicates that for the aforementioned four cases, the average relative error of the eigenvalues by the PCR element method is almost one order of magnitude smaller than that by the CR element method and the conforming linear element method. As is shown in Table \ref{tab:15}- \ref{tab:60}, for a penalty parameter near to the quasi--optimal one, the corresponding PCR element method has similar results.  This means that the PCR element method is in some sense robust with respect to the penalty parameter.

It is observed from Table \ref{tab:15}, when the ratio $\eta$ is fixed, say $\eta=15\%$, the penalty parameters to minimize the average relative errors of the first $\eta$ percent of eigenvalues, on the meshes $\cT_4$, $\cT_5$, $\cT_6$, $\cT_7$, lie in the intervals $[0.7,0.9]$, $[0.7,0.9]$, $[0.6,0.8]$, $[0.6,0.8]$, respectively. This implies that the optimal penalty parameters on different meshes are quite close, which suggests that we only need to compute the quasi-optimal ones on relative coarse meshes. On the other hand, when the mesh is fixed, say $\cT_7$, the optimal penalty parameters to minimize the average relative errors  for Case 1 - 4, lie in the intervals $[0.6,0.8]$, $[0.8,1]$, $[0.9,1.1]$ and $[1,1.2]$, respectively. We can see that the quasi-optimal penalty parameters $\gamma^{\ast}=0.6641$, $0.8203$, $1.0938$, $ 1.2793$ obtained by the improved Algorithm \ref{alg:bisection} either belong or very close to the corresponding intervals, which shows the efficiency of our method.

In our algorithm, the criterion $\textbf{tol}$ is up to our disposal. Table \ref{tab:tolcompare} presents the quasi-optimal penalty parameters computed by the improved Algorithm \ref{alg:bisection} with the fixed ratio $\eta$, say  $10\%$, for various criterions $\textbf{tol}$. It shows that the penalty parameters with various criterions are almost the same, and the average relative error by the PCR element method with the corresponding penalty parameters are also quite similar and close to the minimal one. This indicates that our improved Algorithm \ref{alg:bisection} is robust with respect to the criterion $\textbf{tol}$.

\begin{table}[htbp]
\footnotesize
  \centering
  \caption{\footnotesize{Penalty parameter selected by the Algorithm \ref{alg:bisection} with initial interval $[0,10]$, $k=5$, $\eta=10\%$, $\epsilon=0.01$ and the corresponding iteration numbers and average relative errors of first $15\%$ eigenvalues on $\cT_7$.}}
    \begin{tabular}{c|ccccccccc}
    \hline
     $\textbf{tol}$ & $0.1$ &$0.2$ & $0.3$ & $0.4$ & $0.5$ & $0.6$ & $0.7$ & $0.8$ & $0.9$\\\hline
     $\gamma^{\ast}$& 0.7129 & 0.7129  & 0.7324 & 0.7422 & 0.7227 & 0.6934 & 0.6836& 0.6641 & 0.6936 \\\hline
     $E_{15\%,7}$  & 0.0060 & 0.0060  & 0.0057 & 0.0057 & 0.0058 & 0.0067 & 0.0072& 0.0085 & 0.0067 \\\hline
    \end{tabular}%
  \label{tab:tolcompare}%
\end{table}%

Next, we test the case that all initial guesses are of Type 2. In Table \ref{tab:gammac}, we list six sets of guesses in this case and the corresponding penalty parameter $\hat{\gamma}$ by fitting $\Delta \lambda$ with a quadratic function. In this implementation, we take $\textbf{tol}=0.4$, $k=5$, $\eta=0.15$. It is shown that these penalty parameters $\hat{\gamma}$ of Type 3 are close to the quasi-optimal one, and the corresponding PCR element method also yields high accuracy for the eigenvalues.

\begin{table}[htbp]
\footnotesize
  \centering
  \caption{\footnotesize{For some cases that all initial guesses belong to Type 2, penalty parameter $\hat\gamma$ is selected by fitting the average relative error.}}
    \begin{tabular}{c|cccccc}
    \hline
      set of guesses &$\{0,\ 0.1,\ 0.2\}$&$\{0,\ 0.2,\ 0.4\}$ &$\{0,\ 0.3,\ 0.6\}$ &$\{0,\ 0.4,\ 0.6\}$ &$\{0,\ 0.5 ,\ 0.6\}$  & $\{0.4,\ 0.5,\ 0.6\}$\\\hline
      real roots of $\Delta \lambda$&Yes &Yes &Yes &No & No &No\\\hline
      $\hat{\gamma}$ & 1.9612  &2.1953 &1.3995  &1.7349  &0.8585  &0.8578 \\\hline
      Type of $\hat{\gamma}$& Type 1& Type 1& Type 1& Type 1& Type 3& Type 3\\\hline
      $E_{15\%,7} $& 0.0085&0.0085&0.0085&0.0085&0.0136& 0.0135\\\hline
    \end{tabular}%
  \label{tab:gammac}%
\end{table}%

\subsection{}\label{sec:Lshape}
Next we consider the following eigenvalue problem
\begin{align*}
-\Delta u\ &=\ \lambda u\ \quad \text{  in}\ \Om ,\\
u\ &=\ 0 \qquad \text{ on }\partial \Om ,
\end{align*}
on a L-shaped domain $\Om\ =\ (-1,1)^2/[0,1]\times [-1,0]$. In the computation, the level one mesh is obtained by dividing the domain into three unit squares, each of which is further divided into two triangles. Each mesh is refined into a half sized mesh uniformly, to get a higher level mesh.
\begin{figure}[!ht]
\centering
\subfigure
{\includegraphics[width=6cm,height=6cm]{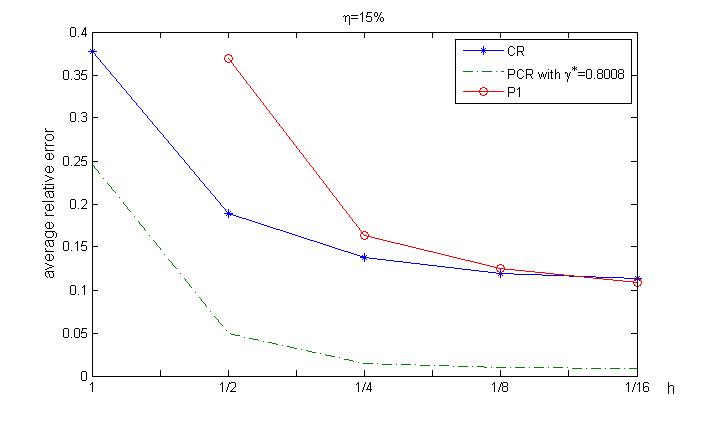}}
\subfigure
{\includegraphics[width=6cm,height=6cm]{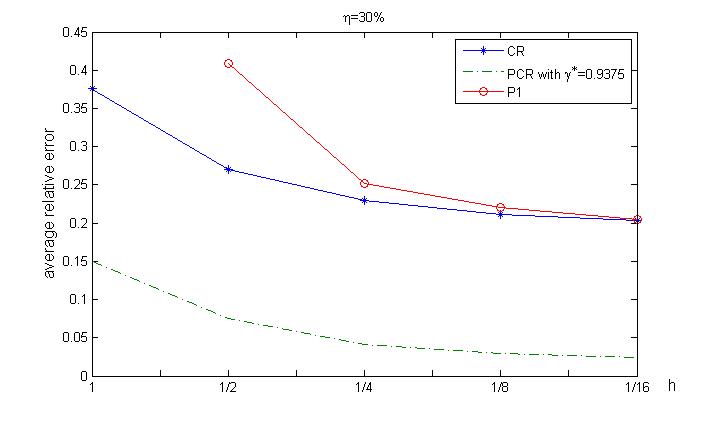}}
\subfigure
{\includegraphics[width=6cm,height=6cm]{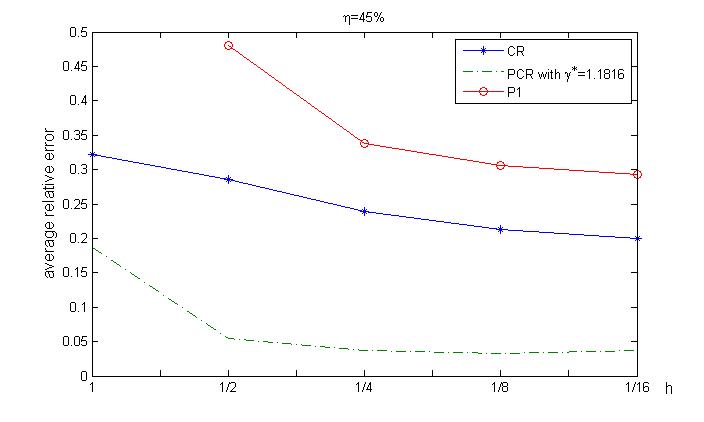}}
\subfigure
{\includegraphics[width=6cm,height=6cm]{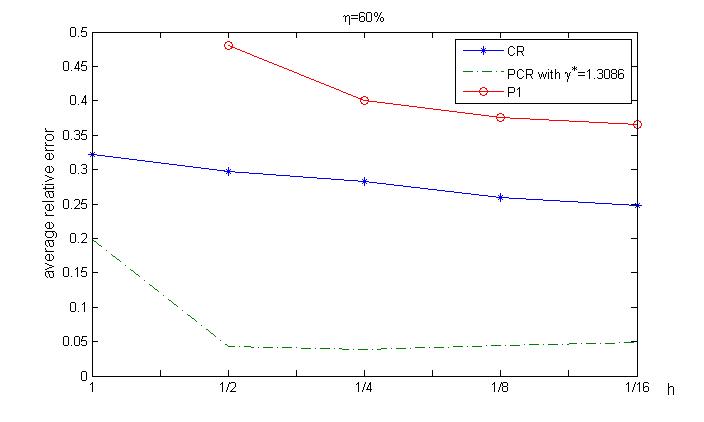}}
\caption{\footnotesize{Average relative errors $E_{\eta,k}$ computed by the CR element method, the conforming linear and the PCR element method with penalty parameters $\gamma^\ast$ on the L-shaped domain. The penalty parameters $\gamma^\ast$ are generated by Algorithm \ref{alg:bisection} with respect to various ratios $\eta$.}}
\label{fig:Lshape}
\end{figure}

Since the exact eigenvalues of this problem are unknown, we apply the high accuracy post processing algorithm proposed in \cite{HuHuangShen2012} to the conforming $P_3$ method and the Weak Element method with $m=3,\ k=5$ on the mesh $\cT_4$, and use the first $23\%$ eigenvalues as the reference eigenvalues.

In our computation, we consider four cases that the ratios of eigenvalues we investigate are $15\%$, $30\%$, $45\%$, $60\%$, denoted by Case 1, Case 2, Case 3 and Case 4, respectively. To select quasi-optimal penalty parameters for these four cases, we let $\eta=10\%$, $20\%$, $30\%$, and $40\%$ in \eqref{betadef}, respectively. In the improved version of Algorithm \ref{alg:bisection}, we take $k=4$, $\textbf{tol}=0.5$, $\epsilon=0.01$. In particular, this yields $\gamma^{\ast}=0.8008$, $0.9375$, $ 1.1816$ and $1.3086$ for these four cases, respectively.

Figure \ref{fig:Lshape} compares the average relative errors by the CR element method, the conforming linear method and the PCR element method with selected penalty parameters $\gamma^{\ast}$ for the four aforementioned cases. Similar to the results in Section \ref{sec:2dtest}, the average relative errors of the approximate eigenvalues by the PCR element method are much better than those by the CR element method and the conforming linear method for all the four cases. In particular, on the fixed mesh $\cT_5$, for Case 1, the average relative error by the PCR element method is smaller than $0.86\%$, a remarkable improvement on $11\%$ by the CR element method; for Case 2, the average relative error by the PCR element method is smaller than $2.42\%$, a big improvement on $20\%$ by the CR element method; for Case 3, the average relative error by the PCR element method is smaller than $3.72\%$, a significant improvement on $20\%$ by the CR element method; for Case 4, the average relative error by the PCR element method is smaller than $4.82\%$, a prominent improvement on $24\%$ by the CR element method.

We plot the reference eigenvalues and the approximate ones by the CR element method and the PCR element method with the penalty parameter $\gamma^{\ast}=1.3086$ on $\cT_4$ in Figure \ref{fig:Lshapecompare}. Similar to the results in Section \ref{sec:2dtest}, the PCR element method with penalty parameter $\gamma^{\ast}=1.3086$ on $\cT_4$ eliminates the emergence that the second third of the approximate eigenvalues by the CR element method are almost the same, therefore, achieves higher accuracy for the first $60\%$ percent of eigenvalues.
\begin{figure}[!ht]
\centering
\subfigure
{\includegraphics[width=6cm,height=6cm]{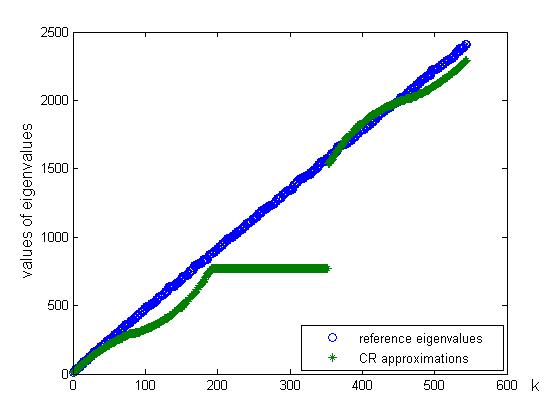}}
\subfigure
{\includegraphics[width=6cm,height=6cm]{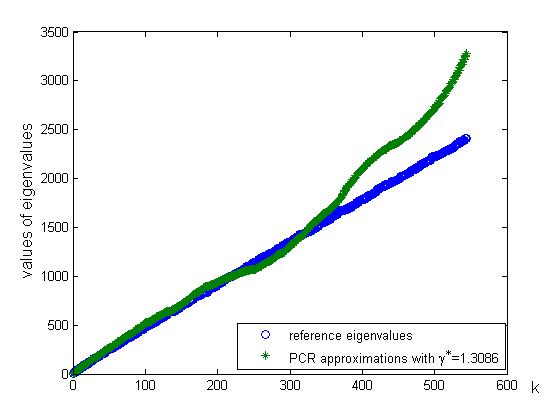}}
\caption{\footnotesize{Comparison of the approximate eigenvalues by the CR element method, the PCR element method with $\gamma^\ast$ with the reference eigenvalues on the L-shaped domain.}}
\label{fig:Lshapecompare}
\end{figure}

\subsection{}
In this experiment, we consider the eigenvalues for a general second order elliptic operator
\begin{equation}\label{numerical:general}
\begin{split}
-\nabla\cdot(a\ \nabla u)\ &=\ \lambda u\ \quad \text{  in}\ \Om ,\\
u\ &=\ 0 \qquad \text{ on }\partial \Om ,
\end{split}
\end{equation}
where $a(x,y)=1+x+y$ and the domain $\Om\ =\ (0,1)^2$ is covered by the uniform triangular mesh in Section \ref{sec:2dtest}.

\begin{table}[!ht]
\footnotesize
  \centering
  \caption{\footnotesize{The average relative errors of the first $15\%$ percent of the eigenvalues by the CR element method, the conforming linear element method and the PCR element method with $\gamma^\ast=1.2012$.}}
    \begin{tabular}{ccccccc}
    \hline
    $h$   &$1$&  $\frac{1}{2}$ &  $\frac{1}{2^2}$ &  $\frac{1}{2^3}$ & $\frac{1}{2^4}$&  $\frac{1}{2^5}$ \\\hline
    P1    & -       & 0.6801 &  0.2441 &   0.1500 &   0.1191 &  0.1102 \\
    CR    &  0.2601 & 0.2214 &   0.1445 &   0.1375 &  0.1273 &0.1228 \\
    PCR   & 0.3862 & 0.1137 & 0.0263 & 0.0183 & 0.0137 & 0.0109 \\\hline
    \end{tabular}%
  \label{tab:change15}%
\end{table}%

\begin{table}[htbp]
\footnotesize
  \centering
  \caption{\footnotesize{The average relative errors of the first $30\%$ percent of the eigenvalues by the CR element method, the conforming linear element method and the PCR element method with $\gamma^\ast=1.5723$.}}
    \begin{tabular}{ccccccc}
    \hline
    $h$   &$1$&  $\frac{1}{2}$ &  $\frac{1}{2^2}$ &  $\frac{1}{2^3}$ & $\frac{1}{2^4}$&  $\frac{1}{2^5}$ \\\hline
    P1    & -       & 0.6801 &	 0.3686 &  0.2489 &  0.2175 & 0.2091\\
    CR    &  0.2601 & 0.2680 &	 0.2341 &  0.2164 &	 0.2054 & 0.1993 \\
    PCR   & 0.4252 & 0.0978 & 0.0572 & 0.0397 & 0.0301 & 0.0259 \\\hline
    \end{tabular}%
  \label{tab:change30}%
\end{table}%

\begin{table}[htbp]
\footnotesize
  \centering
  \caption{\footnotesize{The average relative errors of the first $45\%$ percent of the eigenvalues by the CR element method, the conforming linear element method and the PCR element method with $\gamma^\ast=2.0117$.}}
    \begin{tabular}{ccccccc}
    \hline
    $h$   &$1$&  $\frac{1}{2}$ &  $\frac{1}{2^2}$ &  $\frac{1}{2^3}$ & $\frac{1}{2^4}$&  $\frac{1}{2^5}$ \\\hline
    P1    & -       & 0.6801 &  0.4650 &0.3444 & 0.3059 & 0.2978\\
    CR    &  0.2601 & 0.2901 &	0.2638 &0.2449 & 0.2329 & 0.2265\\
    PCR   & 0.4713 & 0.0585 & 0.0525 & 0.0361 & 0.0272 & 0.0243\\\hline
    \end{tabular}%
  \label{tab:change45}%
\end{table}%

\begin{table}[htbp]
\footnotesize
  \centering
  \caption{\footnotesize{The average relative errors of the first $60\%$ percent of the eigenvalues by the CR element method, the conforming linear element method and the PCR element method with $\gamma^\ast=2.2852$.}}
    \begin{tabular}{ccccccc}
    \hline
    $h$   &$1$&  $\frac{1}{2}$ &  $\frac{1}{2^2}$ &  $\frac{1}{2^3}$ & $\frac{1}{2^4}$&  $\frac{1}{2^5}$ \\\hline
    P1    & -       & 0.6801 & 0.4650 &	 0.3977 &	0.3764 &  0.3681\\
    CR    &  0.2601 & 0.2573 &   0.2816 &	0.2684 &  0.2585 & 0.2523 \\
    PCR   & 0.5000 & 0.0337 & 0.0410 & 0.0343 & 0.0350 & 0.0356  \\\hline
    \end{tabular}%
  \label{tab:change60}%
\end{table}%

The exact eigenvalues of this problem are still unknown, thus, the reference eigenvalues are obtained in the same way as the test in Section \ref{sec:Lshape}. In this implementation, we consider four cases where the ratios of eigenvalues we investigate are $15\%$, $30\%$, $45\%$, and $60\%$, respectively. To select quasi-optimal penalty parameters for these four cases, we let $\eta=10\%$, $20\%$, $30\%$, and $40\%$ in \eqref{betadef}, respectively. We take $k=5$, $\textbf{tol}=0.5$, $\epsilon=0.01$ in the improved version of Algorithm \ref{alg:bisection}. In particular, this yields $\gamma^{\ast}=1.2012$, $1.5723$, $ 2.0117$ and $2.2852$ for these four cases, respectively.

The resulted average relative errors of eigenvalues with respect to the four aforementioned cases are presented in Table \ref{tab:change15}-\ref{tab:change60}. On the mesh $\cT_5$, the average relative errors by the PCR element method for these four cases are smaller than $0.45\%$, $1.31\%$, $3.11\%$ and $4.70\%$, respectively, which is a remarkable improvement compared to the $12\%$, $19\%$, $22\%$ and $25\%$ by the CR element method, respectively. This implies that for a general second order elliptic problem, the PCR element method still shows great advantage over the CR element method and the conforming linear method when a large amount of eigenvalues are investigated.

\subsection{}\label{sec:3dtest}
In this experiment, we consider the following eigenvalue problem
\begin{align}
-\Delta u\ &=\ \lambda u\ \quad \text{  in}\ \Om ,\\
u\ &=\ 0 \qquad \text{ on }\partial \Om ,
\end{align}
on a 3-dimensional domain $\Om\ =\ (0,1)^3$. To obtain the level one mesh, the domain is divided into six tetrahedrons as showed in Figure \ref{fig:3dtriangulation}. Each mesh is refined into a half sized mesh uniformly, to get a higher level mesh. The exact eigenvalues are known as
$$
\lambda=(m^2+n^2+l^2)\pi^2,\ m,\ n,\ l\ \text{are positive integers}
$$
\begin{figure}[!ht]
\centering
\includegraphics[width=6cm,height=6cm]{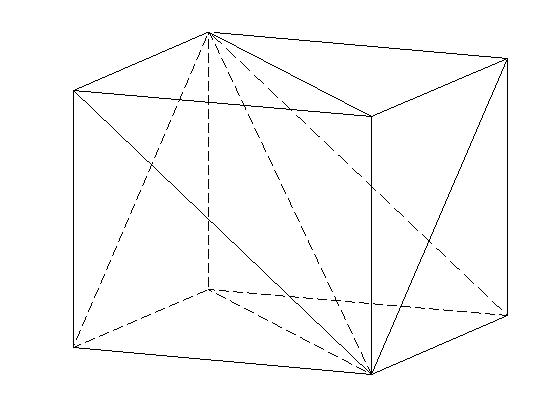}
\caption{\footnotesize{Triangulation in 3-dimension}}
\label{fig:3dtriangulation}
\end{figure}

\begin{figure}[!ht]
\centering
\subfigure
{\includegraphics[width=6cm,height=6cm]{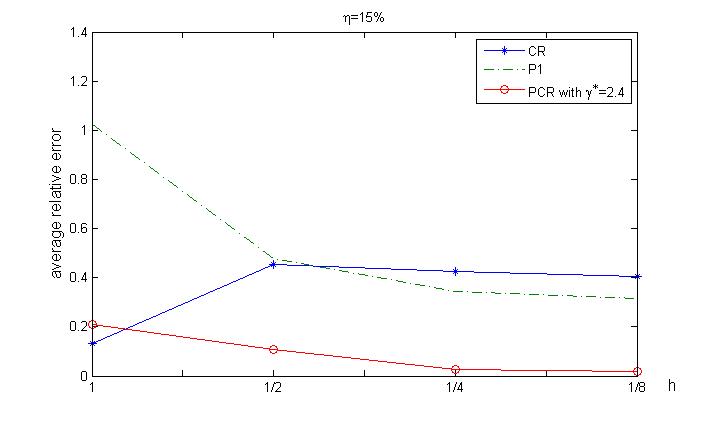}}
\subfigure
{\includegraphics[width=6cm,height=6cm]{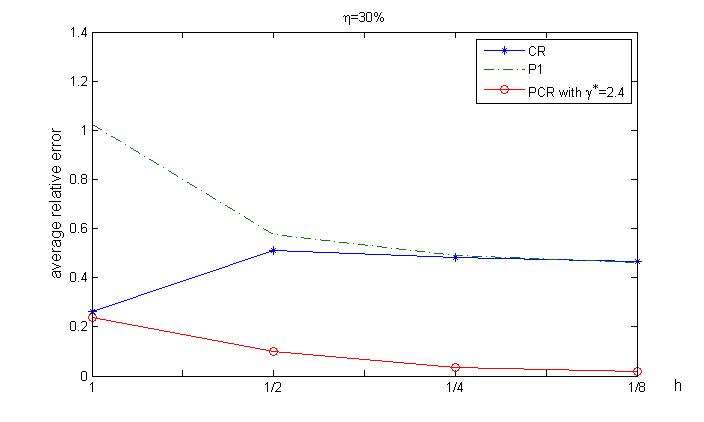}}
\subfigure
{\includegraphics[width=6cm,height=6cm]{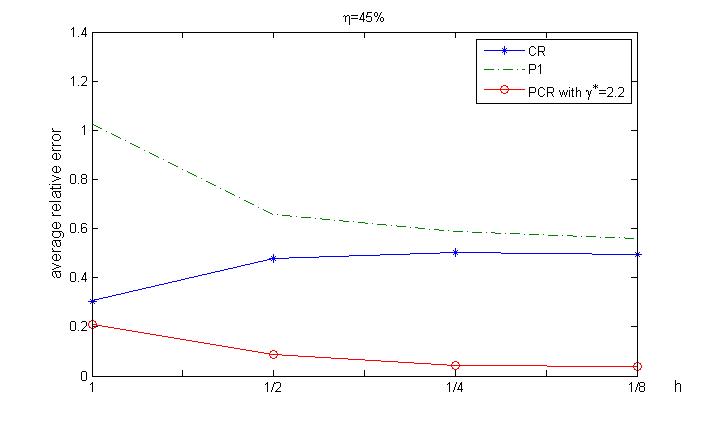}}
\subfigure
{\includegraphics[width=6cm,height=6cm]{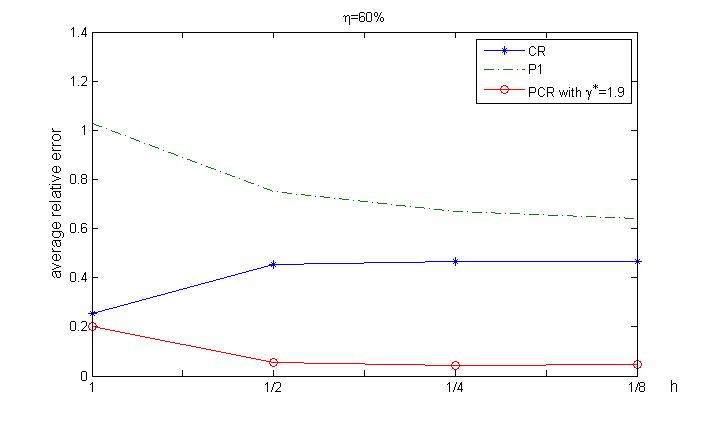}}
\caption{\footnotesize{Average relative errors $E_{\eta,k}$ by the CR element method, the conforming linear element method and the PCR element methods with various penalty parameters on an unit cube $[0,1]^3$.}}
\label{fig:3d}
\end{figure}

In this implementation, we consider four cases where the ratios of eigenvalues we investigate are $15\%$, $30\%$, $45\%$, and $60\%$. For these four cases, we let the penalty parameter be $\gamma^\ast=2.5$, $2.5$, $2.2$ and $2.0$. In Figure \ref{fig:3d}, we compare the average relative errors of the approximate eigenvalues by the CR element method, the conforming linear element method and the PCR element method with penalty parameters $\gamma^\ast$. Similar to the results showed in the two-dimensional numerical examples in Section \ref{sec:2dtest}, with the ratio $\eta$ and mesh fixed, the average relative errors by the PCR element method with the optimal penalty parameters are almost one order of magnitude smaller than the CR element method and the conforming linear element method.

\end{document}